\documentclass[11pt]{article}
\def\date{Mar 4 2019. Last revised Sep 24 2019}
\usepackage{amsmath, amssymb, amsfonts, amsthm, color, appendix, algorithm, algorithmic}

\newtheorem{proposition}{Proposition}[section]
\newtheorem{lemma}[proposition]{Lemma}
\newtheorem{corollary}[proposition]{Corollary}
\newtheorem{theorem}[proposition]{Theorem}
\newtheorem{claim}[proposition]{Claim}

\newtheorem{thm}[proposition]{Theorem}

{
\theoremstyle{definition}
\newtheorem{definition}[proposition]{Definition}

}

\newcommand{\mcal}{\mathcal}

\newcommand{\F}{\mcal{F}}

\newcommand{\R}{\mcal{R}}

\begin{document}
\font\smallrm=cmr8

\baselineskip=12pt
\phantom{a}\vskip .25in
\centerline{{\bf  Linear-Time and Efficient Distributed Algorithms for List Coloring Graphs on Surfaces}}
\vskip.4in
\centerline{{\bf Luke Postle}
\footnote{\texttt{lpostle@uwaterloo.ca. Canada Research Chair in Graph Theory. Partially supported by NSERC under Discovery Grant No. 2019-04304, the Ontario Early Researcher Awards program and the Canada Research Chairs program.}}} 
\smallskip
\centerline{Department of Combinatorics and Optimization}
\smallskip
\centerline{University of Waterloo}
\smallskip
\centerline{Waterloo, Ontario, Canada}

\vskip.25in \centerline{\bf ABSTRACT}
\bigskip

{
\parshape=1.0truein 5.5truein
\noindent

In 1994, Thomassen proved that every planar graph is $5$-list-colorable. In 1995, Thomassen proved that every planar graph of girth at least five is $3$-list-colorable. His proofs naturally lead to quadratic-time algorithms to find such colorings. Here, we provide the first linear-time algorithms to find such colorings.

For a fixed surface $\Sigma$, Thomassen showed in 1997 that there exists a linear-time algorithm to decide if a graph embedded in $\Sigma$ is $5$-colorable and similarly in 2003 if a graph of girth at least five embedded in $\Sigma$ is $3$-colorable. Using the theory of hyperbolic families, the author and Thomas showed such algorithms exist for list-colorings. Around the same time, Dvo\v{r}\'ak and Kawarabayashi also provided such algorithms. Moreover, they gave an $O(n^{O(g+1)})$-time algorithm to find such colorings (if they exist) in $n$-vertex graphs where $g$ is the Euler genus of the surface. Here we provide the first such algorithm which is fixed parameter tractable with genus as the parameter; indeed, we provide a linear-time algorithm to find such colorings. 

In 1988, Goldberg, Plotkin and Shannon provided a deterministic distributed algorithm for $7$-coloring $n$-vertex planar graphs in $O(\log n)$ rounds.  In 2018, Aboulker, Bonamy, Bousquet, and Esperet provided a deterministic distributed algorithm for $6$-coloring $n$-vertex planar graphs in $O(\log^3 n)$ rounds. Their algorithm in fact works for $6$-list-coloring. They also provided an $O(\log^3 n)$-round algorithm for $4$-list-coloring triangle-free planar graphs. Chechik and Mukhtar independently obtained such algorithms for ordinary coloring in $O(\log n)$ rounds, which is best possible in terms of running time. Here we provide the first polylogarithmic deterministic distributed algorithms for $5$-coloring $n$-vertex planar graphs and similarly for $3$-coloring planar graphs of girth at least five. Indeed, these algorithms run in $O(\log n)$ rounds, work also for list-colorings, and even work on a fixed surface (assuming such a coloring exists).
}

\vfill \baselineskip 11pt \noindent \date.
\vfil\eject
\baselineskip 12pt

\section{Introduction}

\noindent {\bf List Coloring Planar Graphs.} Graph coloring is a widely studied area in graph theory. In particular, coloring planar graphs - and more generally coloring graphs embedded in surfaces - has received much attention. Recall that a \emph{$k$-coloring} of a graph $G$ is an assignment of colors $\{1,\ldots,k\}$ such that adjacent vertices do not receive the same color. In 1977, Appel and Haken~\cite{AH1,AH2} proved the Four Color Theorem that every planar graph has a $4$-coloring. In 1989, they~\cite{AHAlg} gave a quartic-time algorithm to find such a coloring, which was improved to a quadratic-time algorithm by Robertson, Sanders, Seymour and Thomas~\cite{RSST} in 1996. In 1959, Gr\"otzsch proved that every triangle-free planar graph $G$ has a $3$-coloring. Thomassen's short proof~\cite{Thom3Short} of Gr\"otzsch's theorem can easily be converted to a quadratic-time algorithm to find such a coloring. This was improved to $O(v(G)\log v(G))$ by Kowalik~\cite{Kow} in 2004 and finally to a linear-time algorithm in 2009 by Dvo\v{r}\'ak, Kawarabayashi and Thomas~\cite{DvoKawTho}.

List-coloring is a generalization of coloring introduced by Erd\H{o}s, Rubin and Taylor~\cite{ERT} in 1979 and independently by Vizing~\cite{Vizing} in 1976. A \emph{list-assignment} $L$ of a graph $G$ is an assignment of lists of colors $L(v)$ to each vertex of $v$ of $G$; we say $L$ is a \emph{$k$-list-assignment} if $|L(v)|\ge k$ for every vertex $v$. An \emph{$L$-coloring} $\phi$ is a coloring such that $\phi(v)\in L(v)$ for each vertex $v$. We say a graph $G$ is \emph{$k$-list-colorable} if $G$ has an $L$-coloring for every $k$-list-assignment $L$. In 1994, Thomassen~\cite{Thom5List} proved that every planar graph is $5$-list-colorable; this is best possible in that Voigt~\cite{Voigt} constructed a planar graph that is not $4$-list-colorable. In 1995, Thomassen~\cite{Thom3List} proved that every planar graph having girth at least five is $3$-list-colorable; this is best possible as Voigt~\cite{VoigtGirthFour} constructed a triangle-free planar graph that is not $3$-list-colorable. Of course, every triangle-free planar graph is $4$-list-colorable as such graphs are $3$-degenerate. 

For the sake of convenience when stating theorems, we say a list assignment $L$ of a graph $G$ is \emph{type $345$} if $L$ is a $k$-list-assignment for some $k\in\{3,4,5\}$ and $G$ has girth at least $8-k$. Thomassen's proofs naturally lead to quadratic-time algorithms to find $L$-colorings for type $345$ list-assignments. It seems conceivable that a clever algorithmic implementation of Thomassen's proofs would yield linear-time algorithms; however, no such algorithm has appeared in the literature to date.

Our first main result (a special case of our more general result about surfaces stated later) is to provide the first linear-time algorithms to find such colorings.

\begin{thm}\label{PlanarLinear}
There exists a linear-time algorithm to find an $L$-coloring of a planar graph $G$ if $L$ is a type 345 list assignment for $G$.
\end{thm}

{\noindent \bf Coloring and List-Coloring Graphs on Surfaces.} It is natural to wonder if the results mentioned above can be extended to graphs embedded in a fixed surface. Of course, there do exist graphs that are not $5$-list-colorable. Since coloring is a monotone property, it is natural to study minimally non-colorable graphs. To that end, we say a graph $G$ is \emph{critical for $k$-coloring} if $G$ is not $k$-colorable but every proper subgraph of $G$ is. Similarly we say a graph $G$ is \emph{critical for $k$-list-coloring} if there exists a $k$-list-assignment $L$ such that $G$ does not have an $L$-coloring but every proper subgraph of $G$ does. 

In a lengthy breakthrough work in 1997, Thomassen~\cite{Thom5Surface} proved that for every surface $\Sigma$, there exist only finitely many graphs embeddable in $\Sigma$ that are critical for $5$-coloring. As a corollary, it follows that there exists a linear-time algorithm for each surface $\Sigma$ to decide if a graph embeddable in $\Sigma$ is $5$-colorable. Similarly in 2003 in yet another deep work, Thomassen~\cite{Thom3Surface} proved that for every surface $\Sigma$, there exist only finitely many graphs of girth at least five embeddable in $\Sigma$ that are critical for $3$-coloring. Hence there exists a linear-time algorithm for each surface $\Sigma$ to decide if a graph of girth at least five embeddable in $\Sigma$ is $3$-colorable.

Thomassen~\cite{Thom5Surface} naturally asked if such results could be generalized to list-coloring. For $5$-list coloring, the author and Thomas~\cite{PostleThomas} proved the analogous result while the author~\cite{Postle} proved the analogous result for $3$-list-coloring graphs having girth at least five. These proofs are rather lengthy and rely on the theory of strongly hyperbolic families developed by the author and Thomas in~\cite{PostleThomas}. Moreover, we proved that such critical graphs have at most $O(g(\Sigma))$ vertices, where $g(\Sigma)$ denotes the \emph{Euler genus} of $\Sigma$ (that is $2h+c$ where $h$ is the number of handles of $\Sigma$ and $c$ is the number of crosscaps). These results imply linear-time algorithms to decide if graphs embeddable in a fixed surface are $5$-list-colorable or $3$-list-colorable if they have girth at least five. Indeed, our algorithms can also decide if an $L$-coloring exists for a fixed type $345$-list-assignment $L$.

As for finding such an $L$-coloring (if it exists) given a fixed type $345$ list-assignment $L$, Dvo\v{r}\'ak and Kawarabayashi~\cite{DvoKawSODA} used some of the results on hyperbolic families combined with some results about the treewidth of graphs on surfaces to provide an $O(v(G)^{O(g(\Sigma)+1)})$-time algorithm to find an $L$-coloring of a graph embedded in $\Sigma$, provided such a coloring exists. Dvo\v{r}\'ak and Kawarabayashi~\cite{DvoKawSODA} asked whether the dependence on $g(\Sigma)$ in the exponent could be removed. This would imply that finding a $5$-list-coloring of a graph (if it exists) and similarly finding a $3$-list-coloring for a girth five graph (if it exists) is fixed parameter tractable (FPT) where the parameter is the genus. We answer this question in the affirmative. Indeed, we generalize Theorem~\ref{PlanarLinear} as follows.

\begin{thm}\label{SurfaceLinear}
For each surface $\Sigma$, there exists a linear-time algorithm to find an $L$-coloring (if it exists) of a graph $G$ embeddable in $\Sigma$ if $L$ is a type 345 list assignment for $G$.
\end{thm}

We note that there are many graphs on surfaces which do have an $L$-coloring for every type $345$ list assignment. In particular, the author and Thomas~\cite{PostleThomas} showed that graphs with edge-width $\Omega(\log g(\Sigma))$ have this property. 



\vskip12pt

{\noindent \bf Deterministic Distributed Algorithms for Coloring and List-Coloring Planar Graphs.} Another paradigm of studying for coloring algorithms is distributed algorithms. Here we study the synchronous message-passing model of distributed computing, in particular the LOCAL model (where unlimited size messages are allowed). However, we believe our algorithm should also work in the CONGEST model (where there are limits on message sizes) but do not pursue this direction of study. Moreover, we restrict our attention to the study of deterministic algorithms in the LOCAL model. In these models, round complexity is the benchmark. Indeed, efficiency is usually defined as having round complexity at most polylogarithmic in the number of vertices. We refer the reader to the survey book of Barenboim and Elkin~\cite{BEBook} for more details on distributed computing. 

What then is known for deterministic distributed algorithm for coloring planar graphs? In 1988, Goldberg, Plotkin and Shannon~\cite{GPS} provided a deterministic distributed algorithm for $7$-coloring $n$-vertex planar graphs in $O(\log n)$ rounds. It should be noted that Goldberg, Plotkin and Shannon gave an $O(\log n \log^* n)$ round parallel algorithm for $5$-coloring $n$-vertex planar graphs. They falsely claimed that all of the results in that paper carried over to the distributed setting but this turns out not to be the case for $5$-coloring planar graphs. Rather their algorithm for $5$-coloring planar graphs would give $O(n)$ round complexity in the distributed setting. This inaccuracy seems to have been perpetuated (e.g. in~\cite{BESublog}) until the authors in~\cite{ABBE} noticed the flaw.  Hence the question of improving the distributed result for $5$-coloring and even $6$-coloring planar graphs had in fact been open. 

In 2018, Aboulker, Bonamy, Bousquet, and Esperet~\cite{ABBE} provided a deterministic distributed algorithm for $6$-coloring $n$-vertex planar graphs in $O(\log^3 n)$ rounds. Their algorithm in fact works for $6$-list-coloring. They also provided an $O(\log^3 n)$-round algorithm for $4$-list-coloring triangle-free planar graphs. Chechik and Mukhtar~\cite{CheMuk} independently obtained such algorithms for ordinary coloring in $O(\log n)$ rounds, which is best possible in terms of round complexity. Aboulker et al.~\cite{ABBE} asked whether such efficient algorithms exist for $5$-coloring planar graphs and $3$-coloring planar graphs of girth at least five. 

Aboulker et al.~\cite{ABBE} proved that no such efficient algorithms exist for $4$-coloring planar graphs and $3$-coloring triangle-free planar graphs. This is because such algorithms would imply that graphs whose $\Omega(\log n)$ neighborhoods are planar are $4$-colorable (and $3$-colorable if triangle-free) and there exist graphs which do not satisfy these properties. They thus also noted that their conjectures for the remaining cases would then imply that $\Omega(\log n)$ locally planar graphs are $5$-colorable and $3$-colorable if girth at least five. That result is true but the only known proof of it follows from the author and Thomas' general result~\cite{PostleThomas} about hyperbolic families.

Here we provide the first polylogarithmic deterministic distributed algorithms for $5$-coloring $n$-vertex planar graphs and similarly for $3$-coloring $n$-vertex planar graphs of girth at least five. Indeed, these algorithms run in $O(\log n)$ rounds and work even for list-colorings. In addition, they even work on a fixed surface assuming a coloring exists.

\begin{thm}\label{SurfaceLOCAL}
For each surface $\Sigma$, there exists a deterministic distributed algorithm that given a graph $G$ embeddable in $\Sigma$ and a type 345 list-assignment $L$ for $G$, finds an $L$-coloring of $G$ (if it exists) in $O(\log v(G))$ rounds.
\end{thm}


\section{Proof Overview and Algorithm}  

In this section, we first provide an overview of the proofs of Theorems~\ref{SurfaceLinear} and~\ref{SurfaceLOCAL} at a high level. Then in the remaining subsections of this section, we proceed to give the formal definitions and statements of the major theorems used in their proofs. We also state the algorithm used for both Theorem~\ref{SurfaceLinear} and Theorem~\ref{SurfaceLOCAL}. Finally, we provide an outline of the rest of the paper at the end of the section.

\vskip.1in

\noindent {\bf Proof Overview.} We provide the same algorithm (Algorithm~\ref{Alg1}, stated in Section~\ref{ColAlg}) for both Theorems~\ref{SurfaceLinear} and~\ref{SurfaceLOCAL}. The complexity analysis for the two theorems is of course different but in both cases is rather straightforward (see Section~\ref{Sec:Algorithm} for details). The proof of the correctness, however, relies on the deep theorems proved by the author and Thomas described in the introduction. Nevertheless, there are two key new contributions in this work. 

The first contribution is a structure theorem (Theorem~\ref{PlanarStructure}) for planar graphs. Theorem~\ref{PlanarStructure} roughly says that a planar graph either has a $2$-separation with a constant-size side, or has a constant-size subgraph that is `deep' (i.e. the ratio between its vertices and its neighbors is large). More than that, the theorem guarantees linearly many pairwise non-touching such subgraphs. We note that Theorem~\ref{PlanarStructure} is not about coloring and hence may be of independent interest for developing efficient algorithms. Indeed, this structure theorem is essential for establishing the running time in both algorithms. 

The second contribution is to combine this structure theorem with the deep coloring theorems proved elsewhere to prove a startling new coloring theorem (Theorem~\ref{DeletablePockets}). Theorem~\ref{DeletablePockets} roughly says that for every planar graph $G$, there exists a constant-size, constant-degree subgraph that is `deletable' for type $345$ list assignments (i.e. every coloring of the remaining graph extends to a coloring of the subgraph no matter the list assignment). More than that, the theorem guarantee linearly many pairwise non-touching such subgraphs. We note that our notion of deletable (see Definition~\ref{DeletableDef}) is quite strong, essentially the strongest such notion one could define.

The algorithm then consists in finding and deleting such deletable subgraphs, then recursively finding a coloring of the smaller graph and finally extending that coloring to the deleted subgraphs. Since these subgraphs are constant-size and constant-degree these searches, deletions and extensions can be performed very efficiently. Since there are linearly many such subgraphs, the smaller graph will be proportionally smaller (i.e. $(1-\varepsilon)v(G)$ for some fixed $\varepsilon$). Combined these two facts lead to a linear-time algorithm (or a $O(\log n)$-round algorithm in the distributed case). 

We note that in the distributed case some care is needed when extending the colorings to guarantee $O(\log n)$-round complexity as opposed to $O(\log n \log^* n)$-round complexity. Namely that while the algorithm recurses, the deleted vertices must compute a constant coloring of an auxiliary graph on the deletable subgraphs needed for extending the coloring efficiently. Such a coloring can be computed in $\tilde{O}(\sqrt{\Delta(G)}) + O(\log^* n)$-round (where $\Delta(G)$ denotes the maximum degree of $G$) using the algorithm of Fragniaud, Heinrich and Kosowski~\cite{FHK}. Since $\Delta(G)$ will be a constant in our application, this will run in $O(\log^* n)$ time. To avoid $O(\log n \log^* n)$-round complexity, though, it is crucial that this coloring be computed while the algorithm recurses so that the overall complexity is $O(\log n)+O(\log^* n) = O(\log n)$-rounds.

In the remainder of this section, we present the formal versions of Theorems~\ref{PlanarStructure} and~\ref{DeletablePockets} and then state Algorithm~\ref{Alg1}. In the remaining sections of the paper, we then present the proof of the theorems and algorithm complexity.

\subsection{A Structure Theorem for Planar Graphs}\label{StrThm}

Here we state our structure theorem for planar graphs but first some definitions.

\begin{definition}
Let $G$ be a graph, $H$ be an induced subgraph of $G$ and $C, k >0$. The \emph{boundary} of $H$ in $G$ is $\{v\in V(H): N(v)\setminus V(H)\ne \emptyset\}$. The \emph{coboundary} of $H$ in $G$ is $\bigcup_{v\in H}N(v) \setminus V(H)$ (i.e. the boundary of $G[V(G)\setminus V(H)]$). 

\noindent We say $H$ is 
\begin{itemize}
\item a \emph{purse} if $H$ is connected, the size of the coboundary of $H$ (call it $S$) is at most two, and $G[S\cup V(H)] + \{uv: u\ne v\in S\}$ is planar,
\item a \emph{$C$-pocket} if $H$ is connected, $v(H)\le C$ and $d_G(v)\le C$ for every $v\in V(H)$,
\item a \emph{$C$-purse} if $H$ is a purse and a $C$-pocket,
\item \emph{$k$-deep} if the coboundary of $H$ is non-empty and has size at most $\frac{1}{k}v(H)$. 
\end{itemize}

\noindent We say two subgraphs $H_1,H_2$ of $G$ are \emph{non-touching} if $V(H_1)\cap V(H_2)=\emptyset$ and there does not exist $a_1\in V(H_1), a_2\in V(H_2)$ such that $a_1a_2\in E(G)$.
\end{definition}

As described in the proof overview, we will prove a structure theorem that every planar graph has a $C$-purse or a $k$-deep $C$-pocket. More than that, we will prove it has linearly many pairwise non-touching such subgraphs. Hence the following definition.

\begin{definition}
Let $G$ be a graph and $C>0$. A \emph{$C$-wallet} of $G$ is a set $\mathcal{H}$ such that all of the following hold:

\begin{itemize}
\item every $H\in \mathcal{H}$ is a $C$-pocket, and
\item every pair $H_1\ne H_2 \in \mathcal{H}$ are non-touching, and 
\item $|\{ H\in \mathcal{H}\}|\ge \frac{1}{2C}v(G)$.
\end{itemize}

\noindent Let $k>0$. We say a $C$-wallet $\mathcal{H}$ is \emph{$k$-deep} if for every $H\in \mathcal{H}$, $H$ is a purse or $k$-deep. 
\end{definition}

We are now prepared to state our structure theorem.

\begin{thm}\label{PlanarStructure}
For every $k\ge 1$, there exists $C>0$ such that the following holds: if $G$ is a planar graph, then there exists a $k$-deep $C$-wallet $\mathcal{H}$ of $G$.
\end{thm}

We can extend Theorem~\ref{PlanarStructure} more generally to surfaces under the assumption the graph is large as follows.

\begin{thm}\label{SurfaceStructure}
For every $k\ge 1$ and surface $\Sigma$, there exists $C>0$ such that the following holds: if $G$ is a graph embeddable in $\Sigma$ with $v(G)\ge Cg(\Sigma)$, then there exists a $k$-deep $C$-wallet $\mathcal{H}$ of $G$.
\end{thm}

\subsection{A Coloring Theorem}\label{ColThm}

Here we state our new coloring theorem but first we define our notion of deletable as follows.

\begin{definition}\label{DeletableDef}
We say an induced subgraph $H$ of a graph $G$ is \emph{$r$-deletable} in $G$ if for every list-assignment $L$ of $H$ such that $|L(v)|\ge r - (d_G(v)-d_H(v))$ for each $v\in V(H)$, $H$ has an $L$-coloring.
\end{definition}

\begin{definition}
We say a $C$-wallet $\mathcal{H}$ of a graph $G$ is $r$-deletable if every $H\in \mathcal{H}$ is $r$-deletable in $G$.
\end{definition}

We will prove that in graphs embedded on surfaces having girth at least $g\in\{3,4,5\}$, purses and $k$-deep pockets (for some large enough $k$) contain $(8-g)$-deletable subgraphs (Lemmas~\ref{PurseDeletable} and~\ref{DeepDeletable} respectively).  Combined with Theorem~\ref{SurfaceStructure}, we obtain the following result which we now state as our main coloring theorem as follows.

\begin{theorem}\label{DeletablePockets}
For each surface $\Sigma$, there exists $C > 0$ such that the following holds:
If $G$ is a graph embeddable in $\Sigma$ having girth at least $g\in \{3,4,5\}$ and $v(G)\ge Cg(\Sigma)$, then there exists an $(8-g)$-deletable $C$-wallet of $G$.
\end{theorem}

Theorem~\ref{DeletablePockets} is quite surprising. It implies that a minimum counterexample to the strong hyperbolicity of $5$-list-coloring or $3$-list-coloring graphs of girth five (e.g. for Theorem~\ref{OldStronglyHyperbolic}) has constant size and hence that theorem can be proved by brute-force methods. It also implies that minimum counterexamples to Thomassen's theorems that planar graphs are $5$-list-colorable and that planar graphs of girth at least five are $3$-list-colorable have constant size and hence that those theorems can also be proved by brute-force methods (instead of clever inductive arguments). 

\subsection{A Coloring Algorithm}\label{ColAlg}

Before stating our algorithm, we need one definition as follows.
\begin{definition}
If $H$ is a graph and $k\ge 1$ is an integer, then $H^k$ denotes the graph with $V(H^k)=V(H)$ and $E(H^k) = \{u\ne v\in V(H): u$ is at distance at most $k$ from $v\}$.
\end{definition}

\newpage

\noindent We are now ready to state our algorithm for both Theorem~\ref{SurfaceLinear} and~\ref{SurfaceLOCAL}:
\vskip.1in

\noindent Let $\Sigma$ be a fixed surface, let $C_0$ be as in Theorem~\ref{DeletablePockets} for $\Sigma$ and let $C=\max\{C_0,2\}$.
\vskip.1in

\begin{algorithm}
\caption{Finding an $L$-coloring of a graph $G$ embedded in $\Sigma$ for a type $345$ list-assignment $L$}
\label{Alg1}
\begin{algorithmic}[1]
\vskip.05in
\STATE {\bf if} $v(G)\le Cg(\Sigma)$ {\bf then} find and {\bf return} an $L$-coloring of $G$ by exhaustive search\\
\vskip.05in
\STATE {\bf for} each vertex $v$ of $G$ of degree at most $C$: \\
\vskip.05in
\hspace{10mm}Find (if one exists) an $(8-g)$-deletable $C$-pocket $H_v$ containing $v$ \\
\vskip.05in
Let $H=\bigcup_v H_v$
\vskip.05in
\STATE {\bf recurse} on $G_0 = G-V(H)$ to find an $L$-coloring $\psi_0$ of $G_0$
\vskip.05in
\STATE While recursing, also find (i.e. simultaneously for LOCAL) a $C^{2C}$-coloring $\phi$ of $H^{2C}$
\vskip.05in
\STATE {\bf for} each $i\in \{1, \ldots, C^{2C}\}$:
\begin{enumerate}
\item[(i)] Let $H_i = (\bigcup_v H_v: \phi(v)=i)$
\item[(ii)] Restrict the coloring $\psi_{i-1}$ to $V(G_{i-1}) \setminus V(H_i)$
\item[(iii)] Extend the restriction to a coloring $\psi_i$ of $G_i = G[V(G_{i-1})\cup V(H_i)]$ 
\end{enumerate}
\vskip.05in
\STATE {\bf return} $\psi_{C^{2C}}$
\end{algorithmic}
\end{algorithm}

\subsection{Outline of Paper}

In Section~\ref{Sec:Algorithm}, we prove Theorem~\ref{SurfaceLinear} and Theorem~\ref{SurfaceLOCAL}, namely, by proving the correctness and runtime complexity of Algorithm~\ref{Alg1}. In Section~\ref{Sec:Structure}, we prove Theorem~\ref{SurfaceStructure} and hence also its special case Theorem~\ref{PlanarStructure}. Finally in Section~\ref{Sec:Deletable}, we prove Theorem~\ref{DeletablePockets}.

\section{Algorithm Analysis}\label{Sec:Algorithm}

In this section, we prove the correctness, termination and runtime complexity of Algorithm~\ref{Alg1}. First, we need the following lemma to lower bound the size of $H$ in Step 2. Crucially, we need Theorem~\ref{DeletablePockets} to do this.

\begin{lemma}\label{SmallerRecursiveGraph}
Let $\Sigma$ be a surface and let $C$ be as in Theorem~\ref{DeletablePockets} for $\Sigma$. Let $G$ be a graph embeddable in $\Sigma$ having girth at lest $g\in\{3,4,5\}$ such that $v(G)\ge Cg(\Sigma)$. If $S$ is the set of all vertices in $G$ that are contained in an $(8-g)$-deletable $C$-pocket, then $|S| \ge \frac{v(G)}{2C}$.
\end{lemma}
\begin{proof}
Since $v(G)\ge Cg(\Sigma)$, we have by Theorem~\ref{DeletablePockets} that there exists an $(8-g)$-deletable $C$-wallet $\mathcal{H}$ of $G$. By definition then $|\{H \in \mathcal{H}\}| \ge \frac{1}{2C}v(G)$. Hence $\sum_{H\in \mathcal{H}} v(H) \ge \frac{1}{2C}v(G)$. Yet $\bigcup_{H\in \mathcal{H}} V(H)\subseteq S$. Since the subgraphs in $\mathcal{H}$ are pairwise vertex-disjoint, it follows that $|S|\ge \frac{1}{2C}v(G)$.
\end{proof}

Lemma~\ref{SmallerRecursiveGraph} implies an upper bound on $G_0$ in Step 2 as follows.

\begin{corollary}\label{ProportionallySmaller}
If $G$ is as in Algorithm~\ref{Alg1} and $G_0$ is defined as in Step 2, then $v(G_0)\le \left(1-\frac{1}{2C}\right)v(G)$.
\end{corollary}
\begin{proof}
Since $V(H)$ is precisely the set of vertices contained in an $(8-g)$-deletable $C$-pocket, we find by Lemma~\ref{SmallerRecursiveGraph} that $v(H)\ge \frac{1}{2C}v(G)$. Hence $v(G_0)\le v(G)-v(H) \le \left(1-\frac{1}{2C}\right)v(G)$ as desired.
\end{proof}

We are now ready to prove the termination and correctness of Algorithm~\ref{Alg1}.

\begin{lemma}\label{AlgorithmTerminationCorrectness}
Algorithm~\ref{Alg1} terminates and returns an $L$-coloring of $G$.
\end{lemma}
\begin{proof}
We prove this by induction on $v(G)$. If $v(G)\le Cg(\Sigma)$ as tested in Step 1, then the algorithm returns a coloring by exhaustive search. That search runs in constant time. Note that for planar graph an $L$-coloring always exists since $L$ is a type $345$ list-assignment (this follows from the results of Thomassen mentioned in the introduction; see Lemma~\ref{PurseDeletable} for more details). As for graphs embedded in surfaces other than the plane, we assumed that $G$ has an $L$-coloring and hence so does every subgraph of $G$. 

So we may assume $v(G)\ge Cg(\Sigma)$. Step 2 is simply definitional for the purposes of terminations and correctness. By Corollary~\ref{ProportionallySmaller}, $v(G_0)\le \left(1-\frac{1}{2C}\right)v(G) < v(G)$. Hence by induction, Step 3 returns an $L$-coloring $\psi_0$ of $G_0$.

As for Step 4, we note that $\Delta(H)\le C$ since every vertex in $H$ has degree at most $C$ in $G$. Since $C\ge 2$, it follows that $\Delta(H^{2C}) < C^{2C}$. Thus $H^{2C}$ has a $C^{2C}$-coloring using the greedy bound of $\Delta(H^{2C})+1$. 

As for Step 5, we now prove by induction that for each $i\in \{1,\ldots, C^{2C}\}$ there exists an $L$-coloring $\psi_i$ of $G_i$ as in Step 5(iii) whose restriction to $V(G_{i-1})\setminus V(H_i)$ agrees with $\psi_{i-1}$. Let $u\ne v \in V(H)$ such that $\phi(u)=\phi(v)=i$. Since $v(H_u),v(H_v)\le C$ and the distance between $u$ and $v$ in $G$ is at least $2C+1$ as $\phi(u)=\phi(v)$, we have that $H_u$ and $H_v$ are non-touching. Thus the set of subgraphs $\mathcal{H} = \{H_v: \phi(v) =i\}$ are pairwise non-touching. 

Now we define a list assignment $L_i$ of $H_i$ as follows. For each $v\in V(H_i)$, let $L_i(v) = L(v)\setminus \{\psi_{i-1}(u): u\in N(v)\cap (V(G_{i-1})\setminus V(H_i))\}$. Now $|L_i(v)|\ge |L(v)| - (d_G(v) - d_{H_i}(v)) = 8-g - (d_G(v) - d_{H_i}(v))$. Let $H_v\in \mathcal{H}$. Since $H_v$ is $(8-g)$-deletable in $G$, it follows by definition that there exists an $L_i$-coloring $\psi_v$ of $H_v$. Now we define $\psi_i$ as follows: let $\psi_i(w) = \psi_{i-1}(w)$ if $w\in V(G_{i-1})\setminus V(H_i)$ and otherwise let $\psi_i(w) = \psi_v(w)$ where $w\in H_v$. By the definition of $L_i$ and since the $H_v$ are pairwise non-touching, it follows that $\psi_i$ is an $L$-coloring of $G_i$ as desired.

Hence in Step 6, Algorithm~\ref{Alg1} returns $\psi_{C^{2C}}$ which is an $L$-coloring of $G_{C^{2C}} = G$ as desired.
\end{proof}

Next we prove that Algorithm~\ref{Alg1} runs in linear-time in the centralized setting as follows.

\begin{lemma}\label{AlgorithmLinear}
Algorithm~\ref{Alg1} runs in $O(v(G))$-time.
\end{lemma}
\begin{proof}
We proceed by induction on $n=V(G)$. We prove the runtime is $an$ for some constant $a$ as follows. Specifically $a=2CB$ where $B$ is as defined below.

We note that Step 1 runs in constant time. As for Step 2, we note that deciding if a given constant-size subgraph is $r$-deletable can be done in constant time (since it is necessary to only a check a finite number of list assignments). Thus for each vertex $v$, deciding if an $H_v$ exists and finding one if it does can be done in constant time (since there are only a constant number of candidates as they have constant-size, have constant-degree in $G$, are connected and contain $v$). Hence Step 2 runs in $O(n)$ time (admittedly for some very large constant). 

Finding $G_0$ for use in Step 3 can also be done in linear time (since planar graphs have at most $3n$ edges). Similarly for Step 4, finding $H$ and $H^{2C}$ can be done in linear time (since $H$ has maximum degree at most $C$). 

Step 4 then runs in $C^{2C}n$ time since $\Delta(H^{2C}) < C^{2C}$ and a $(\Delta(H)+1)$-coloring of $H$ can be found in $(\Delta(H)+1)n$ time by greedy coloring. As for Step 5, for each $i\in \{1,\ldots, C^{2C}\}$, steps 5(ii) and 5(iii) can be done in linear time. Hence step 5 runs in linear-time.
 
Altogether each run of the algorithm (not counting recursion) is linear-time, specifically at most $Bn$ for some universal constant $B$ (depending only on $C$). By Corollary~\ref{ProportionallySmaller}, $v(G_0) \le \left(1-\frac{1}{2C}\right)n$. Let $r=1-\frac{1}{2C}$. By induction, Step 3 takes $a|V(G_0)|=a(rn)$ time. Hence the total runtime of the algorithm is at most $Bn + arn = (B+ar)n$ which is at most $an$ as desired if $B+ar \le a$, that is, $a \ge \frac{B}{1-r}$.
\end{proof}

Next we prove the runtime of Algorithm~\ref{Alg1} in the deterministic distributed setting.

\begin{lemma}\label{AlgorithmLOCAL}
Algorithm~\ref{Alg1} runs in $O(\log v(G))$-rounds in the deterministic distributed setting.
\end{lemma}
\begin{proof}
We proceed by induction on $n=V(G)$.

Step 1 runs in $Cg(\Sigma)$ rounds (i.e.~a constant number). Step 2 runs in $C$ rounds (i.e.~a constant number). For Step 4, we use the deterministic distributed algorithm of Fragniaud, Heinrich and Kosowski~\cite{FHK} to find a $(\Delta(H^{2C})+1)$-coloring of $H^{2C}$ in $\tilde{O}(\sqrt{\Delta}) + O(\log^* n)$ rounds. Since $\Delta(H)< C^{2C}$, this yields a $C^{2C}$-colorng $\phi$ of $H^{2C}$ as desired. The number of rounds required for Step 3 is thus $\tilde{O}(C^{C}) + O(\log^* n) = O(\log^* n)$. As for Step 5, for each $i\in \{1,\ldots, C^{2C}\}$, steps 5(i)-(iii) run in $C$ rounds. Hence Step 5 runs in $C^{2C+1}$ rounds (i.e.~a constant number). 

By Corollary~\ref{ProportionallySmaller}, $v(G_0) \le \left(1-\frac{1}{2C}\right)n$. Thus the recursive step (Step 3) takes by induction $O(\log(v(G_0)) = O(\left(1-\frac{1}{2C}\right)n)$ rounds. Since Step 4 is performed separately while the recursion runs, the total runtime is $O(\log^* n)+O(\log n) = O(\log n)$ rounds.
\end{proof}

Finally, we have that Theorem~\ref{SurfaceLinear} and~\ref{SurfaceLOCAL} follow immediately from the above lemmas.

\begin{proof}[Proof of Theorem~\ref{SurfaceLinear}]
Follows from Lemmas~\ref{AlgorithmTerminationCorrectness} and~\ref{AlgorithmLinear}.
\end{proof}

\begin{proof}[Proof of Theorem~\ref{SurfaceLOCAL}]
Follows from Lemmas~\ref{AlgorithmTerminationCorrectness} and~\ref{AlgorithmLOCAL}.
\end{proof}

\section{Surface Structure Theorem}\label{Sec:Structure}

Theorems~\ref{PlanarStructure} and its generalization Theorem~\ref{SurfaceStructure} are mostly straightforward consequences of the following theorem of Lipton and Tarjan~\cite{LipTar} about the separator hierarchy.

\begin{lemma}\label{Shatterer}
For every proper minor-closed family $\mathcal{F}$ and every $\varepsilon > 0$, there exists a constant $C$ such that the following holds:
If $G$ is a graph in $\F$, then there exists $X\subseteq V(G)$ with $|X|\le \varepsilon v(G)$ such that every component of $G-X$ has size at most $C$.
\end{lemma}

The proof of the result is essentially iteratively finding sublinear separators and taking $X$ to be their union, wherein we terminate the procedure right before $X$ becomes too large, whence the size of components in $G-X$ will be a small constant depending on $\varepsilon$.

We recall that for every proper minor-closed family $\mathcal{F}$, there exists $c_{\mathcal{F}}$ such that for every graph $G\in \mathcal{F}$, we have that $e(G) \le c_{\mathcal{F}} v(G)$. That is, graphs in proper minor-closed families have a linear number of edges. Such a result was first shown by Mader~\cite{Mader1,Mader2}. Using this, we obtain the following proposition.

\begin{proposition}\label{LowDeg}
For every proper minor-closed family $\mathcal{F}$ and every $\varepsilon > 0$, there exists a constant $C$ such that the following holds:
If $G$ is a graph in $\F$, then there exists $X\subseteq V(G)$ with $|X|\le \varepsilon v(G)$ such that every $v\in V(G)\setminus X$ has degree at most $C$ in $G$.
\end{proposition}
\begin{proof}
Let $C = \frac{2c_{\mathcal{F}}}{\varepsilon}$. Let $X=\{v\in V(G): d(v) > C\}$. Note that $2e(G) = \sum_{v\in V(G)} d(v) > |X|C$. Yet $e(G) \le c_{\mathcal{F}}v(G)$ for some constant $c_{\mathcal{F}}$ as noted above. Hence $|X| \le \frac{2e(G)}{C} \le \frac{2}{C} (c_{\mathcal{F}}v(G)) = \varepsilon v(G)$ as desired.
\end{proof}

Combining Lemma~\ref{Shatterer} with Proposition~\ref{LowDeg} yields the following corollary.

\begin{corollary}\label{LowDegShatterer}
For every proper minor-closed family $\mathcal{F}$ and every $\varepsilon > 0$, there exists a constant $C$ such that the following holds:
If $G$ is a graph in $\F$, then there exists $X\subseteq V(G)$ with $|X|\le \varepsilon v(G)$ such that every component of $G-X$ is a $C$-pocket of $G$.
\end{corollary}
\begin{proof}
Let $C_1$ be as in Lemma~\ref{Shatterer} for $\mathcal{F}$ and $\frac{\varepsilon}{2}$. By Theorem~\ref{Shatterer}, there exists $X_1\subseteq V(G)$ with $|X_1|\le \frac{\varepsilon}{2} v(G)$ such that every component of $G-X_1$ has size at most $C_1$. 

Let $C_2$ be as in Proposition~\ref{LowDeg} for $\mathcal{F}$ and $\frac{\varepsilon}{2}$. By Proposition~\ref{LowDeg}, there exists $X_2\subseteq V(G)$ with $|X_2|\le \frac{\varepsilon}{2} v(G)$ such that every $v\in V(G)\setminus X_2$ has degree at most $C_2$.

Let $C=\max\{C_1,C_2\}$. Let $X=X_1\cup X_2$. Now $|X|\le |X_1|+|X_2| \le \varepsilon v(G)$. Moreover, every component of $G-X$ is a subgraph of a component of $G-X_1$ and hence has size at most $C_1 \le C$. On the other hand, every vertex in $V(G)\setminus V(X)$ is a vertex in $V(G)\setminus V(X_2)$ and hence has degree at most $C_2 \le C$ in $G$. Thus every component of $G-X$ is a $C$-pocket of $G$ and hence $X$ is as desired.
\end{proof}

Next we prove a crucial lemma which shows how common purses are. To that end, recall the following definitions.

\begin{definition}
A \emph{vertex (resp. edge) amalgamation} of two graphs $G_1$ and $G_2$ is a graph obtained by identifying a vertex (resp. edge) in $G_1$ with a vertex in $G_2$.
\end{definition}

Thus a vertex amalgamation is simply a $1$-sum, while an edge amalgamation is a $2$-sum where the edge is retained. 

\begin{definition}
The \emph{(Euler) genus} of graph $G$ is the minimum $k$ such that $G$ embeds in a surface of Euler genus $k$.
\end{definition}

We need the following result of Miller~\cite{Miller}.

\begin{theorem}[Miller~\cite{Miller}]\label{AdditiveGenus}
Euler genus is additive under vertex and edge amalgamations.
\end{theorem}

For ease of reading in the proof, we adopt the following notation: if $\mathcal{H}$ is a set of subgraphs of a graph $G$, we let $V(\mathcal{H})=\bigcup_{H\in\mathcal{H}} V(H)$, let $v(\mathcal{H})=|V(\mathcal{H})|$ and finally let $|\mathcal{H}|$ denote the number of subgraphs in $\mathcal{H}$.

\begin{lemma}\label{ManyPursesOrKDeep}
If $G$ is a connected graph embeddable in a surface $\Sigma$ and $X\subseteq V(G)$, then there are most $12k|X|+16kg(\Sigma)$ vertices in $G-X$ that are in components of $G-X$ that are neither purses of $G$ nor $k$-deep in $G$.
\end{lemma}
\begin{proof}
Let $\mathcal{H}_0 = \{H: H$ is a component of $G-X: H$ is neither a purse nor $k$-deep$\}$. Consider the following partition of $\mathcal{H}_0$:
\begin{itemize}
\item $\mathcal{H}_1 = \{H\in \mathcal{H}_0:$ the coboundary of $H$ has size one$\}$,
\item $\mathcal{H}_{2} = \{H\in \mathcal{H}_0:$ the coboundary of $H$ has size two$\}$,
\item $\mathcal{H}_3 = \{H\in \mathcal{H}_0:$ the coboundary of $H$ has size at least three$\}$.
\end{itemize}

\noindent We further partition $\mathcal{H}_2$ as follows:

\begin{itemize}
\item $\mathcal{H}_{2,1} = \{H\in\mathcal{H}_2:$ there does not exist $H'\ne H \in \mathcal{H}_2$ with the same coboundary as $H\}$,
\item $\mathcal{H}_{2,2} = \mathcal{H}_2 \setminus \mathcal{H}_{2,1}$.
\end{itemize}

\begin{claim}\label{Nonplanar}
$$v(\mathcal{H}_1)+v(\mathcal{H}_{2,2}) \le 4kg(\Sigma).$$
\end{claim}
\begin{proof}
Let $P$ be the set of unordered pairs $x\ne y \in V(G)$ such that there exists $H\in\mathcal{H}_{2,2}$ with coboundary $\{x,y\}$. For each $\{x,y\} \in P$, choose one $H_{xy}\in\mathcal{H}_2$ such that $H_{xy}$ has coboundary $\{x,y\}$. 

Let $G_1 = G[X\cup V(\mathcal{H}_1) \cup V(\mathcal{H}_{2,2})]$. Note that $G_1$ is a subgraph of $G$. Let $G_1'$ be obtained from $G_1$ by for each $\{x,y\}\in P$, deleting $H_{xy}$ and adding the edge $xy$ to $G$. Note that $G_1'$ is a minor of $G_1$ and hence is a minor of $G$. Thus $G_1'$ has Euler genus at most $g(\Sigma)$.

Let $H=G[X] + \bigcup_{\{x,y\}\in P} xy$. Now $G_1'$ is the vertex/edge amalgamations of $H$ and $|\mathcal{H}_1|+|\mathcal{H}_{2,2}|-|P|$ non-planar graphs. Since $G_1'$ has Euler genus at most $g(\Sigma)$ and non-planar graphs have Euler genus at least one, it follows from Theorem~\ref{AdditiveGenus} that

$$ |\mathcal{H}_1|+|\mathcal{H}_{2,2}|-|P| \le g(\Sigma).$$

\noindent Note that $2|P| \le |\mathcal{H}_{2,2}|$ and hence $2|\mathcal{H}_1| + |\mathcal{H}_{2,2}| \le 2g(\Sigma)$. 

Let $H_1\in\mathcal{H}_1$. Note that $H_1$ has a coboundary of size one. Yet $H_1$ is not $k$-deep. Hence $|H_1|\le k$. Thus $v(\mathcal{H}_1) \le k |\mathcal{H}_1|$. Similarly, we find that $v(\mathcal{H}_{2,2})\le 2k|\mathcal{H}_{2,2}|$ and the claim follows.
\end{proof}

\begin{claim}\label{Twos}
$$ v(\mathcal{H}_{2,1}) \le 6k(|X|+g(\Sigma)).$$
\end{claim}
\begin{proof}
Let $P'$ be the set of unordered pairs $x\ne y \in V(G)$ such that there exists $H\in\mathcal{H}_{2,1}$ with coboundary $\{x,y\}$. Let $G_2 = X + \bigcup_{\{x,y\}\in P'} xy$. Note that $G_2$ is a minor of $G$ and hence has Euler genus at most $g(\Sigma)$. By Euler's formula, $e(G_2)\le 3(|X|+g(\Sigma))$. Hence, 

$$|\mathcal{H}_{2,1}| = e(G_2) \le 3(|X|+g(\Sigma)).$$

Since every element of $\mathcal{H}_{2,1}$ has a coboundary of size two and is not $k$-deep, we have that $v(\mathcal{H}_{2,1}) \le 2k|\mathcal{H}_{2,1}|$ and the claim follows.
\end{proof}

\begin{claim}\label{Threes}
$$v(\mathcal{H}_3) \le 6k(|X|+g(\Sigma)).$$
\end{claim}
\begin{proof}
Let $G_3$ be obtained from $G[X \cup V(\mathcal{H}_3)]$ by contracting each $H\in\mathcal{H}_3$ to a new vertex $v_H$ and deleting loops and parallel edges. Now $G_3$ is a minor of $G$ and hence has Euler genus at most $g(\Sigma)$. Note that $G_3$ is bipartite. By Euler's formula, $e(G_3) \le 2(v(G_3)+g(\Sigma))$. Yet every vertex in $V(G_3)\setminus X$ has degree at least three. Hence $e(G_3) \ge 3|\mathcal{H}_3|$. Since $v(G_3) = |X|+|\mathcal{H}_3|$, we find that 

$$|\mathcal{H}_3| \le 2(|X|+g(\Sigma)),$$

\noindent and hence

$$e(G_3) \le 6(|X|+g(\Sigma)).$$

\noindent Let $H\in\mathcal{H}_3$. Note that $v_H$ has degree in $G_3$ equal to the size of the coboundary of $H$ in $G$. Moreover, since $H$ is not $k$-deep, it follows that $v(H) \le k \cdot {\rm deg}_{G_3}(v_H)$. Since $G_3$ is bipartite, it follows that $e(G_3) = \sum_{H\in\mathcal{H}_3} {\rm deg}_{G_3}(v_H)$ and hence

$$v(\mathcal{H}_3) \le k \cdot e(G_3) \le 6k(|X|+g(\Sigma)),$$

\noindent as claimed.
\end{proof}

Combining Claims~\ref{Nonplanar},~\ref{Twos} and~\ref{Threes}, we find that

$$v(\mathcal{H}_0) \le 12k|X| + 16kg(\Sigma),$$

\noindent as desired.
\end{proof}

We are now ready to prove Theorem~\ref{SurfaceStructure} of which Theorem~\ref{PlanarStructure} is a special case. 

\begin{proof}[Proof of Theorem~\ref{SurfaceStructure}]
Let $\mathcal{F}$ be the family of graphs embeddable in $\Sigma$. Recall that $\mathcal{F}$ is a minor-closed family. Let $C_0$ be as $C$ in Corollary~\ref{LowDegShatterer} for $\mathcal{F}$ with $\varepsilon=\frac{1}{58k}$. Since $k\ge 1$, it follows that $\varepsilon < 1$ and hence $C_0\ge 1$. Let $C=58kC_0$. Since $v(G)\ge Cg(\Sigma)$ by assumption, we have that $g(\Sigma)\le \frac{1}{58kC_0}v(G)$ which is at most $\frac{1}{58k}v(G)$ since $C_0\ge 1$.

By Corollary~\ref{LowDegShatterer}, there exists $X\subseteq V(G)$ with $|X|\le \varepsilon v(G)$ such that every component of $G-X$ is a $C_0$-pocket. Let $\mathcal{H}_1 = \{H: H$ is a component of $G-X$, $H$ is neither a purse nor $k$-deep$\}$. By Lemma~\ref{ManyPursesOrKDeep}, we have that

$$v(\mathcal{H}_1) \le 12k|X| + 16kg(\Sigma).$$

\noindent Let $\mathcal{H} = \{H$ is a component of $G-X: H \notin \mathcal{H}_1\}$. Now

$$v(\mathcal{H}) \ge v(G) - |X| - v(\mathcal{H}_1) \ge v(G) - (12k+1)|X| - 16kg(\Sigma).$$

\noindent Since $k\ge 1$, we have that $12k+1\le 13k$. Since $|X| \le \varepsilon v(G) = \frac{1}{58k}v(G)$ and $g(\Sigma)\le \frac{1}{58k}v(G)$, we find that

$$v(\mathcal{H}) \ge v(G) - (13k+16k)\cdot \frac{v(G)}{58k} \ge \frac{1}{2} v(G).$$  

\noindent Since every $H\in\mathcal{H}_2$ has at most $C$ vertices, it follows that

$$|\mathcal{H}|\ge \frac{1}{2C}v(G),$$

\noindent and hence $\mathcal{H}$ is a $k$-deep $C$-wallet as desired.
\end{proof}

\section{Linear Size for No Deletable Subgraphs}\label{Sec:Deletable}

We prove Theorem~\ref{DeletablePockets} via Theorem~\ref{SurfaceStructure}. Namely Theorem~\ref{SurfaceStructure} guarantees us a $k$-deep $C$-wallet, that is a set of linearly many pairwise non-touching subgraphs, each of which is either a $C$-purse or a $k$-deep $C$-pocket. Thus to prove Theorem~\ref{DeletablePockets}, it suffices to prove the following two lemmas.

\begin{lemma}\label{PurseDeletable}
Let $g\in \{3,4,5\}$. If $D$ is a purse of a graph $G$ having girth at least $g$, then $D$ is $(8-g)$-deletable in $G$.
\end{lemma}

\begin{lemma}\label{DeepDeletable}
For each surface $\Sigma$, there exists $k>0$ such that the following holds: Let $g\in \{3,4,5\}$ and let $G$ be a graph embeddable in $\Sigma$ having girth at least $g$. If $D$ is an induced $k$-deep subgraph of $G$, then there exists $X\subseteq V(D)$ such that $G[X]$ is $(8-g)$-deletable in $G$.
\end{lemma}

In Section 5.1, we establish Lemma~\ref{PurseDeletable} using the standard results of Thomassen on list-coloring planar graphs.

On the other hand, the proof of Lemma~\ref{DeepDeletable} will require the use of the theory of hyperbolic families as well as the deep results of the author and Thomas~\cite{PostleThomas} on the strong hyperbolicity of graphs that are critical for $5$-list-coloring, and of the author~\cite{Postle} on the strong hyperbolicity of graphs of girth five that critical for $3$-list-coloring. We use these and other results to show that embedded graphs having girth at least $g$ without $(8-g)$-deletable subgraphs form a \emph{strongly hyperbolic family} (see Definition~\ref{StronglyHyperbolicDef} below). The Strongly Hyperbolic Structure Theorem (Theorem 7.2 in~\cite{PostleThomas}) will then imply that for every surface there exists a $k$ such that every $k$-deep subgraph of a graph embedded in that surface contains an $(8-g)$-deletable subgraph, thereby completing the proof of Theorem~\ref{DeletablePockets}.

We are now ready to prove Theorem~\ref{DeletablePockets} assuming Lemmas~\ref{PurseDeletable} and~\ref{DeepDeletable}.

\begin{proof}[Proof of Theorem~\ref{DeletablePockets}]
Fix a surface $\Sigma$. Let $k$ be as in Lemma~\ref{DeepDeletable} for $\Sigma$. Let $C$ be as in Theorem~\ref{SurfaceStructure} for $\Sigma$ and $k$. Thus by Theorem~\ref{SurfaceStructure}, there exists a $k$-deep $C$-wallet $\mathcal{H}$ of $G$. Recall that a $C$-wallet is a set of at least $\frac{1}{2C}v(G)$ pairwise non-touching $C$-pockets. Recall that a $C$-wallet is $k$-deep if for every $D\in\mathcal{H}$, either $D$ is a purse or $k$-deep. 

For each $D\in\mathcal{H}$ such that $D$ is a purse, we have by Lemma~\ref{PurseDeletable} that $D$ is $(8-g)$-deletable in $G$. For each $D\in\mathcal{H}$ such that $D$ is $k$-deep, we have by Lemma~\ref{DeepDeletable} that there exists $X_D \subseteq V(D)$ such that $G[X_D]$ is $(8-g)$-deleletable.

Let $\mathcal{H'} = \{D\in\mathcal{H} : D$ is a purse$\} \cup \{X_D: D\in\mathcal{H}$ is $k$-deep$\}$. Thus $\mathcal{H'}$ is a set of at least $\frac{1}{2C}v(G)$ pairwise non-touching subgraphs of $G$, each of which is $(8-g)$-deletable, as desired.
\end{proof}

\vskip.1in

\noindent{\bf Outline of Section.} In Section~\ref{DelPurses}, we prove Lemma~\ref{PurseDeletable}. In Section~\ref{Theory}, we recall the many definitions and theorems of hyperbolic families that we need for the proof of Lemma~\ref{DeepDeletable}. In Section~\ref{KeyFamily}, we define the crucial object of study: a certain family of embedded graphs with no deletable subgraphs and then provide a proof of Lemma~\ref{DeepDeletable} assuming that family is strongly hyperbolic. In Section~\ref{PreviousResults}, we collect the other previous results about hyperbolic families that we need for the proof. In Section~\ref{DeepHyperbolic}, we prove that our crucial family is hyperbolic (see Definition~\ref{HyperbolicDef}; this is a weaker notion than strongly hyperbolic). Finally in Section~\ref{DeepStronglyHyperbolic}, we then use the result of Section~\ref{DeepHyperbolic} to prove that the crucial family is in fact strongly hyperbolic.

\subsection{Deletability for Purses}\label{DelPurses}

In order to prove Theorem~\ref{DeletablePockets} via Theorem~\ref{SurfaceStructure}, we will need to prove that every $C$-purse contains an $(8-g)$-deletable subgraph. In fact, we will prove the stronger theorem that every $C$-purse is itself $(8-g)$-deletable. This mostly follows from theorems of Thomassen as we will show. First we prove for planar graphs that every $L$-coloring of a short precolored path extends to an $L$-coloring of the whole graph using results of Thomassen as follows.

\begin{theorem}\label{PrecoloredPath}
Let $G$ be a planar graph having girth at least $g\in\{3,4,5\}$ and let $L$ be a type $345$ list-assignment for $G$. If $P$ is a path in $G$ such that $v(P)\le g-1$, then every $L$-coloring of $P$ extends to an $L$-coloring of $G$.
\end{theorem}
\begin{proof}
For $g=3$, this follows from Thomassen's stronger inductive theorem about $5$-list-coloring planar graphs (specifically the main Theorem in~\cite{Thom5List}). For $g=5$, this follows from Thomassen's stronger inductive theorem about $3$-list-coloring planar graphs of girth at least five (specifically Theorem 2.1 in~\cite{Thom3Short}, the main result of that paper). 

For $g=4$, this is rather straightforward as follows. We proceed by induction on $v(G)$. We may assume $v(G)\ge 3$ and that $V(P)\ne V(G)$ as otherwise the result is trivial. We may assume by induction that $G$ is connected. 

We claim there exists a vertex $v\in V(G)\setminus V(P)$ of degree at most $3$. Suppose not. By Euler's formula since $v(G)\ge 3$ and $G$ is triangle-free, we have that $e(G) \le 2v(G)-4$. Thus, $\sum_{v\in V(G)} d(v) \le 4v(G)-8$. Yet $\sum_{v\in V(P)}d(v) \ge 4(v(G)-v(P))$. Combining the previous inequalities, we find that $\sum_{v\in V(P)} d_G(v) \le 4v(P)-8$. But $\sum_{v\in V(P)}d_G(v) \ge \sum_{v\in V(P)}d_P(v) = 2v(P)-2$ since $P$ is a path. It follows that $v(P)=3$ and $P$ is a component of $G$, contradicting that $G$ is connected. This proves the claim.

By induction, $\phi$ extends to an $L$-coloring of $G-v$ and hence to an $L$-coloring of $G$ as desired.
\end{proof}

We now use Theorem~\ref{PrecoloredPath} to prove Lemma~\ref{PurseDeletable} as follows.

\begin{proof}[Proof of Lemma~\ref{PurseDeletable}]
Suppose not. That is, there exists a list-assignment $L_0$ such that $|L_0(v)|\ge 8-g - (d_G(v)-d_D(v))$ for each $v\in V(D)$ and $D$ is not $L_0$-colorable.

Let $S$ be the coboundary of $D$ in $G$. First suppose $S=\emptyset$. Then $L_0$ is an $(8-g)$-list-assignment of $D$. By Theorem~\ref{PrecoloredPath} with $P=\emptyset$, there exists an $L_0$-coloring of $D$, a contradiction.

Next suppose $|S|=1$. Let $S=\{x\}$. Let $G'= G[V(D)\cup\{x\}]$. Let $c$ be a new color (that is $c\notin \bigcup_{v\in V(D)} L_0(v)$). Define a new list assignment $L$ of $G'$ as follows. Let $L(x)=\{c\}\cup R$ where $R$ is a set of $7-g$ arbitrary colors. For each $v\in V(D)$, let $L(v) = L_0(v) \cup \{c\}$. 

Now $L$ is an $(8-g)$-list-assignment of $G'$. It follows from Theorem~\ref{PrecoloredPath} with $P=x$ that there exists an $L$-coloring $\phi$ of $G'$ such that $\phi(x)=c$. But then $\phi$ induces an $L_0$-coloring of $D$, a contradiction.

Finally suppose $|S|=2$. Let $S=\{x,y\}$. Let $G'=G[V(D)\cup\{x,y\}]$. Let $G''$ be obtained from $G'$ as follows: delete the edge $xy$ if it exists and add a path $P=xv_1\ldots v_{g-3}y$ with new vertices $v_1,\ldots v_{g-3}$. Now $G''$ has girth at least $g$. Since $D$ is a purse, it follows that $G''$ is planar.

Let $\phi$ be a coloring of $P$ with entirely new colors (i.e.~ $\phi(u)\notin \bigcup_{v\in V(H)} L_0(v)$ for every $u\in V(P)$) such that $\phi(x)\ne \phi(y)$. Define a new list assignment $L$ of $G'$ as follows. For each $u\in V(P)$, let $L(u)=\{\phi(u)\}\cup R$ where $R$ is a set of $7-g$ arbitrary colors. For each $v\in V(D)$, let $L(v) = L_0(v) \cup \{\phi(u): u\in N(v)\cap V(P)\}$. Now $L$ is an $(8-g)$-list-assignment of $G''$. It follows from Theorem~\ref{PrecoloredPath} that $\phi$ extends to an $L$-coloring of $G''$. But then this induces an $L_0$-coloring of $D$, a contradiction.
\end{proof}

\subsection{Theory of Hyperbolic Families}\label{Theory}

We will need to recall a number of definitions from~\cite{PostleThomas}. First we recall the definition of a graph with rings and how they embed in a surface.

\begin{definition}[Graph with Rings - Definition 3.1 in~\cite{PostleThomas}]
A \emph{ring} is a cycle or a complete graph on one or two vertices. A \emph{graph with
rings} is a pair $(G,\R)$, where $G$ is a graph and $\R$ is a set of vertex-disjoint rings in $G$. 
\end{definition}

\begin{definition}[Embedding Graphs with Rings - Definition 3.2 in~\cite{PostleThomas}]
We say that a graph $G$ with rings $\R$ is \emph{embedded in a surface $\Sigma$} if the
underlying graph $G$ is embedded in $\Sigma$ in such a way that for every ring $R \in \R$ there exists a
component $\Gamma$ of the boundary of $\Sigma$ such that $R$ is embedded in $\Gamma$, no other vertex or edge of $G$ is embedded
in $\Gamma$, and every component of the boundary of $\Sigma$ includes some ring of $G$. 
\end{definition}

Now let us state the formal definition of a hyperbolic family. Informally the definition says that for every graph in the family the number of vertices inside a disk is at most linear in the number of vertices on the boundary of that disk.

\begin{definition}[Hyperbolic Family - Definition 5.1 in~\cite{PostleThomas}]\label{HyperbolicDef}
Let $\F$ be a family of non-null embedded graphs with rings. We say that $\F$ is \emph{hyperbolic} if there exists a constant $c>0$ such that if $G\in\F$ is a graph with rings that is embedded in a surface $\Sigma$, then for every closed curve $\gamma: \mathbb{S}^1 \rightarrow \Sigma$ that bounds an open disk $\Delta$ and intersects $G$ only in vertices, if $\Delta$ includes a vertex of $G$, then the number of vertices of $G$ in $\Delta$ is at most $c(|\{x\in \mathbb{S}^1: \gamma(x)\in V(G)\}|-1)$. We say that $c$ is a \emph{Cheeger constant} of $\F$.
\end{definition}

Finally let us also state the definition of a strongly hyperbolic family. Informally the definition extends the linearity property of hyperbolic families from disks to also include annuli (a.k.a.~cylinders).

\begin{definition}[Strongly Hyperbolic Family - Definition 7.1 in~\cite{PostleThomas}]\label{StronglyHyperbolicDef}
Let $\F$ be a hyperbolic family of embedded graphs with rings, let $c$ be a Cheeger constant for $\F$, and let $d := \lceil 3(2c + 1) \log_2(8c + 4)\rceil$. We say that $\F$ is \emph{strongly
hyperbolic} if there exists a constant $c_2$ such that for every $G \in \F$ embedded in a surface $\Sigma$ with rings and for every two disjoint cycles $C_1,C_2$ of length at most $2d$ in $G$, if there
exists a cylinder $\Lambda \subseteq \Sigma$ with boundary components $C_1$ and $C_2$, then $\Lambda$ includes at most $c_2$ vertices of $G$. We say that $c_2$ is a \emph{strong hyperbolic constant} for $\F$.
\end{definition}

A key result we need from~\cite{PostleThomas} is that the number of vertices in a graph in a strongly hyperbolic family is linear in the sum of its genus and ring vertices as follows. Before stating that result, we need one more technical but rather innocuous definition.

\begin{definition}[Closed under curve cutting - appears after Theorem 1.1 in in~\cite{PostleThomas}]
A family $\mathcal{F}$ of embedded graphs is \emph{closed under curve cutting} if for every embedded graph $G \in \mathcal{F}$ embedded in a surface $\Sigma$ and every simple closed curve $\xi$ in $\Sigma$ whose image is disjoint from $G$, if $\Sigma'$ denotes the surface obtained from $\Sigma$ by cutting open along $\xi$ and attaching disk(s) to the resulting curve(s), then the embedded graph
$G$ in $\Sigma'$ belongs to $\mathcal{F}$.
\end{definition}

\begin{theorem}[A Simplified Form of Theorem 7.2 in~\cite{PostleThomas}]\label{StronglyHyperbolicTheorem}
Let $\F$ be a strongly hyperbolic family of embedded graphs with rings such that $\mathcal{F}$ is closed under curve-cutting. Then there exists a constant $c_{\mathcal{F}}$ such the the following holds: if $G \in \mathcal{F}$ is embedded in a surface $\Sigma$ of Euler genus $g$ with a total of $R$ ring vertices, then $v(G) \le c_{\mathcal{F}}(g+R)$.
\end{theorem}

\subsection{A Family with No Deletable Subgraphs}\label{KeyFamily}

Next we define our crucial family of embedded graphs with rings that we need to prove Lemma~\ref{DeepDeletable}.

\begin{definition}\label{Family}
For each $g\in \{3,4,5\}$, let $\mathcal{F}_g$ be the family of embedded graphs with rings $(G,\mathcal{R})$ such that $G-\bigcup_{R\in\mathcal{R}} E(R)$ has girth at least $g$ and there does not exist an $(8-g)$-deletable subgraph of $G$ whose vertices all lie in $G-\bigcup_{R\in\mathcal{R}} V(R)$. Let $\mathcal{F}_{345} = \mathcal{F}_3\cup \mathcal{F}_4\cup \mathcal{F}_5$.
\end{definition}

We note that $\mathcal{F}_{345}$ is clearly closed under curve cutting. Using results by the author and Thomas, we will prove in Section~\ref{DeepStronglyHyperbolic} the following theorem.

\begin{theorem}\label{DeletableStronglyHyperbolic}
$\mathcal{F}_{345}$ is strongly hyperbolic.
\end{theorem}

The following is an immediate corollary of Theorem~\ref{DeletableStronglyHyperbolic} and Theorem~\ref{StronglyHyperbolicTheorem}.

\begin{corollary}\label{DeletableLinear}
There exists $k>0$ such that the following holds: Let $G$ be embedded in a surface $\Sigma$ of genus $g(\Sigma)$ having girth at least $g\in \{3,4,5\}$. If $H$ is a subgraph of $G$ such that there does not exist an induced subgraph of $G-V(H)$ that is $(8-g)$-deletable in $G$, then $v(G) \le k(v(H)+g(\Sigma))$.
\end{corollary}
\begin{proof}
By Theorem~\ref{DeletableStronglyHyperbolic}, we have that $\mathcal{F}_{345}$ is strongly hyperbolic. We let $k=c_{\mathcal{F}_{345}}$ be as in Theorem~\ref{StronglyHyperbolicTheorem}.

Now let $G'=G-E(H)$. We view $G'$ as a graph with rings where each vertex of $H$ is a ring. Now $G'$ also has girth at least $g$ and is embedded in $\Sigma$. Since there does not exist a subgraph of $G-V(H)$ that is $(8-g)$-deletable in $G$, it follows that $G'\in \mathcal{F}_{345}$. Hence by Theorem~\ref{StronglyHyperbolicTheorem}, we have that $v(G)\le k(v(H)+g(\Sigma))$ as desired.
\end{proof}

We now prove Lemma~\ref{DeepDeletable} assuming Theorem~\ref{DeletableStronglyHyperbolic}.

\begin{proof}[Proof of Lemma~\ref{DeepDeletable}]
Let $k_0$ be as in Corollary~\ref{DeletableLinear}. We let $k= k_0(g(\Sigma)+1)+1$. Suppose for a contradiction that there does not exist $X\subseteq V(D)$ such that $G[X]$ is $(8-g)$-deletable in $G$. 

Let $H$ be the coboundary of $D$ in $G$. Let $G'=G[V(H)\cup V(D)]$. Note then that there does not exist an induced subgraph of $G'-V(H)$ that is deletable in $G'$ (as otherwise it would be $(8-g)$-deletable in $G$). Now by Corollary~\ref{DeletableLinear}, we have that $v(G')\le k_0(v(H)+g(\Sigma))$. 

Since $D$ is $k$-deep, $v(H)$ is non-empty by definition and hence $v(H)\ge 1$. Thus $k_0(v(H)+g(\Sigma))< kv(H)$ and so $v(G') < kv(H)$. Rewriting, we have that $v(H) > \frac{1}{k}v(G') > \frac{1}{k}v(D)$. Yet since $D$ is $k$-deep, we have by definition that $v(H) \le \frac{1}{k}v(D)$, a contradiction.
\end{proof}

\subsection{Previous Results on Hyperbolic Families}\label{PreviousResults}

In order to prove Theorem~\ref{DeletableStronglyHyperbolic}, we will first prove that $\mathcal{F}_{345}$ is hyperbolic and then use that in the proof that $\mathcal{F}_{345}$ is strongly hyperbolic. However, we will need a number of previous results for both those proofs which we now collect.

First we recall the definition for being critical with respect to a subgraph for list-coloring.

\begin{definition}
Let $G$ be a graph and let $H$ be a subgraph of $G$. If $L$ is a list assignment of $G$, then we say $G$ is \emph{$H$-critical with respect to $L$} if the following holds: for every proper subgraph $G'$ of $G$ containing $H$, there exists an $L$-coloring of $H$ that extends to $G'$ but not to $G$. We say $G$ is \emph{$H$-critical for $k$-list-coloring} if there exists a $k$-list-assignment such that $G$ is $H$-critical with respect to $L$.
\end{definition}

\begin{definition}
For each $g\in \{3,4,5\}$, let $\mathcal{G}_g$ be the family of embedded graphs with rings $(G,\mathcal{R})$ such that every cycle in $G$ of length at most $(g-1)$ is not null-homotopic and $G$ is $\bigcup_{R\in \mathcal{R}} R$-critical for $(8-g)$-list-coloring. Let $\mathcal{G}_{345} = \mathcal{G}_3\cup \mathcal{G}_4\cup \mathcal{G}_5$.
\end{definition}

It was proved by the author and Thomas in~\cite{PostleThomas2} that $\mathcal{G}_{3}$ is hyperbolic and by Dvo\v{r}\'ak and Kawarabayashi in~\cite{DvoKawSODA} that $\mathcal{G}_{5}$ is hyperbolic (the proof that $\mathcal{G}_{4}$ is hyperbolic is easy by comparison and can be found in~\cite{PostleThomas}). The author and Thomas proved that $\mathcal{G}_{345}$ is strongly hyperbolic as follows.

\begin{theorem}\label{OldStronglyHyperbolic}
$\mathcal{G}_{345}$ is strongly hyperbolic.
\end{theorem}
\begin{proof}
$\mathcal{G}_3$ is strongly hyperbolic by Lemma 7.5 in~\cite{PostleThomas}. $\mathcal{G}_4$ is strongly hyperbolic by Lemma 7.6 in~\cite{PostleThomas}. $\mathcal{G}_5$ is strongly hyperbolic by Lemma 7.8 in~\cite{PostleThomas}. Hence $\mathcal{G}_{345}$ is also strongly hyperbolic.
\end{proof}

We remark that the proof of Lemmas 7.8 in~\cite{PostleThomas} relies on the work of the author in~\cite{Postle} for $3$-list-coloring girth five graphs. Similarly the proof of Lemma 7.5 relies on the work by the author and Thomas in the series of papers~\cite{PostleThomas1,PostleThomas2,PostleThomas3} (see~\cite{PostleThesis} for a full proof). On the other hand, the proof of Lemma 7.6 for $4$-list-coloring triangle-free graphs is once again quite easy by comparison (and can be found in Lemmas 5.10 and 7.6 in~\cite{PostleThomas}). 

Combining Theorem~\ref{OldStronglyHyperbolic} with Theorem~\ref{StronglyHyperbolicTheorem}, we obtain the following corollary.

\begin{corollary}\label{OldLinear}
There exists $k'>0$ such that the following holds: Let $G$ be embedded in a surface $\Sigma$ of genus $g(\Sigma)$ having girth at least $g\in \{3,4,5\}$. If $H$ is a subgraph of $G$ such that $G$ is $H$-critical for $(8-g)$-list-coloring, then $v(G) \le k'(v(H)+g(\Sigma))$.
\end{corollary}
\begin{proof}
By Theorem~\ref{OldStronglyHyperbolic}, we have that $\mathcal{G}_{345}$ is strongly hyperbolic. We let $k'=c_{\mathcal{G}_{345}}$ be as in Theorem~\ref{StronglyHyperbolicTheorem}.

Now let $G'=G-E(H)$. We view $G'$ as a graph with rings where each vertex of $H$ is a ring. Now $G'$ also has girth at least $g$ and is embedded in $\Sigma$. Since $G$ is $H$-critical for $(8-g)$-list-coloring, it follows that $G'$ is also $H$-critical for $(8-g)$-list-coloring. Thus $G'\in \mathcal{G}_{345}$. Hence by Theorem~\ref{StronglyHyperbolicTheorem}, we have that $v(G)\le k'(v(H)+g(\Sigma))$ as desired.
\end{proof}

To prove Theorem~\ref{DeletableStronglyHyperbolic}, we will actually need a stronger result, namely instead of the hyperbolicity of $\mathcal{G}_{345}$, we will need a stronger density result (Theorem~\ref{dPlanar} below). Theorem~\ref{dPlanar} will be the key to proving that $\mathcal{F}_{345}$ is hyperbolic. Finally to prove that $\mathcal{F}_{345}$ is strongly hyperbolic, we will also use Theorem~\ref{dPlanar} along with the strong hyperbolicity of $\mathcal{G}_{345}$ .

To state the density result, we first define a crucial notion of the density of a graph over a subgraph.

\begin{definition}
Let $g\in \{3,4,5\}$ and $\varepsilon > 0$. If $G$ is a graph and $H$ is a subgraph of $G$, define

$$d_{g,\varepsilon}(G|H) = (g-2)(e(G)-e(H)) - (g+\varepsilon)(v(G)-v(H)).$$
\end{definition}

We will also need the following easy proposition about $d_{g,\varepsilon}$.

\begin{proposition}\label{dSum}
Let $g\in \{3,4,5\}$ and $\varepsilon > 0$. If $H\subseteq G' \subseteq G$, then 
$$d_{g,\varepsilon}(G|H) = d_{g,\varepsilon}(G|G')+d_{g,\varepsilon}(G'|H).$$
\end{proposition}
\begin{proof}
This follows from the definition of $d_{g,\varepsilon}$ since $e(G)-e(H)=(e(G)-e(G')) + (e(G')-e(H))$ and $v(G)-v(H) = (v(G)-v(G')) + (v(G')-v(H))$.
\end{proof}

For a planar graph $G$ of girth at least $g$, satisfying $d_{g,\varepsilon}(G|H)\ge 0$ implies an upper bound on $v(G)$ that is linear in $v(H)$ as our next lemma shows (a useful fact for proving the linear bounds required for hyperbolicity and strong hyperbolicity and hence the reason for the introduction of $d_{g,\varepsilon}$).

\begin{lemma}\label{dBound}
Let $g\in \{3,4,5\}$ and $\varepsilon > 0$. If $G$ is a planar graph having girth at least $g$ and $H$ is a subgraph of $G$ such that $d_{g,\varepsilon}(G|H) \ge 0$, then

$$v(G) \le \frac{g+\varepsilon}{\varepsilon} v(H).$$
\end{lemma}
\begin{proof}
We have that $d_{g,\varepsilon}(G|H) = (g-2)(e(G)-e(H)) - (g+\varepsilon)(v(G)-v(H)) \ge 0$. Rearranging, we have that 

$$\varepsilon v(G) \le (g+\varepsilon)v(H) - (g-2)e(H) + ((g-2)e(G) - gv(G)).$$

\noindent Since $G$ is a planar graph having girth at least $g$, we have by Euler's formula that $e(H)\le \frac{g}{g-2}v(G)$. Combining this observation with the fact that $e(H)\ge 0$, we find that
$$v(G) \le \frac{(g+\varepsilon)}{\varepsilon}v(H),$$
as desired.
\end{proof}

We may now state the density result which combines earlier theorems of the author~\cite{Postle} and the author and Thomas~\cite{PostleThomas}. 

\begin{theorem}\label{dPlanar}
There exists $\varepsilon > 0$ such that following holds:
If $G$ is a plane graph having girth at least $g\in\{3,4,5\}$ and $H$ is a connected subgraph of $G$ such that $G$ is $H$-critical for $(8-g)$-list-coloring, then 
$$d_{g,\varepsilon}(G|H)\ge 0.$$
\end{theorem}
\begin{proof}
For $g=3$, this is equivalent to Theorem 4.6 of the author and Thomas in~\cite{PostleThomas2}, which constitutes the main result of that paper. For $g=4$, the proof is straightforward and is implicit in Lemma 5.10 of the author and Thomas in~\cite{PostleThomas}. For $g=5$, this is a special case of Theorem 3.9 of the author in~\cite{Postle}, which again constitutes the main result of that paper.
\end{proof}

We note that combined with Lemma~\ref{dBound}, Theorem~\ref{dPlanar} implies that the family $\mathcal{G}_{345}$ is hyperbolic. This is because essentially hyperbolicity would be equivalent to showing that $v(G)$ is linear in $v(H)$ for connected subgraphs $H$, while instead Theorem~\ref{dPlanar} proves a stronger relation between the edges and vertices of $G$ and $H$. Indeed these stronger statements were necessary for the inductive proofs of the hyperbolicity of $\mathcal{G}_3$ and $\mathcal{G}_5$. 

We note that the assumption that $H$ is connected is related to hybercolity (being inside a disc) while the case of $H$ having two connected components is related to strong hyperbolocity (being inside an annulus).

\subsection{Hyperbolicity for No Deletable Subgraphs}\label{DeepHyperbolic}

We will use Theorem~\ref{dPlanar}, a statement stronger than the hyperbolicity of $\mathcal{G}_{345}$, to prove that the family $\mathcal{F}_{345}$ is hyperbolic. First, we need the following easy proposition.

\begin{proposition}\label{DeletableToCritical}
Let $G$ be a graph and $H$ a proper subgraph of $G$ such that $V(H)\subsetneq V(G)$. If $G-V(H)$ is not $r$-deletable in $G$, then there exists a subgraph $G'$ of $G$ containing $H$ such that $G'$ is $H$-critical for $r$-list-coloring.
\end{proposition}
\begin{proof}
Since $G-V(H)$ is not $r$-deletable in $G$, there exists a list-assignment $L_0$ such that $|L_0(v)|\ge r - d_H(v)$ for each $v\in V(G)\setminus V(H)$ and $G-V(H)$ is not $L_0$-colorable.

Let $S=\bigcup_{u\in V(G)\setminus V(H))} L_0(v)$. Let $S'=\{c_v: v\in V(H)\}$ be a set of pairwise distinct new colors (that is $c_v\notin S$ for each $v\in V(H)$). Let $R$ be an arbitrary set of $r-1$ colors disjoint from $S'$. 

Define a new list assignment $L$ of $G$ as follows. For each $v\in V(H)$, let $L(v)=\{c_v\}\cup R$. For each $u\in V(G)\setminus V(H)$, let $L(u) = L_0(u) \cup \{c_v: v\in N(u)\cap V(H)\}$. Now $L$ is an $r$-list-assignment of $G$. Let $\phi$ be the coloring of $H$ given by $\phi(v)=c_v$ for every $v\in V(H)$.

Since $G-V(H)$ is not $L_0$-colorable, it follows that $\phi$ does not extend to an $L$-coloring of $G$. Let $G'$ be an inclusion-wise minimal subgraph of $G$ containing $H$ such that $\phi$ does not extend to an $L$-coloring of $G'$. By the minimality of $G'$, $\phi$ extends to an $L$-coloring of every proper subgraph of $G'$ containing $H$. Thus $G'$ is $H$-critical with respect to $L$. Hence by definition, $G'$ is $H$-critical for $r$-list-coloring, as desired.
\end{proof}

Using Theorem~\ref{dPlanar}, we are now ready to prove the hyperbolicity of $\mathcal{F}_{345}$ as follows. First we prove a density result as follows.

\begin{lemma}\label{dDeletable}
Let $\varepsilon$ be as in Theorem~\ref{dPlanar}. If $G$ is a plane graph having girth at least $g\in\{3,4,5\}$ and $H$ is a connected subgraph of $G$ such that there does not exist $X\subseteq V(G)\setminus V(H)$ such that $G[X]$ is $(8-g)$-deletable in $G$, then 
$$d_{g,\varepsilon}(G|H)\ge 0.$$
\end{lemma}
\begin{proof}
We proceed by induction on $v(G)-v(H) + e(G)-e(H)$.

If $V(H)=V(G)$, then $d_{g,\varepsilon}(G|H)\ge 0$ as desired. So we may assume that $V(H)\ne V(G)$. By assumption, $G-V(H)$ is not $(8-g)$-deletable in $G$. By Proposition~\ref{DeletableToCritical}, it follows that there exists a subgraph $G'$ of $G$ containing $H$ such that $G$ is $H$-critical for $(8-g)$-list-coloring. Note that $H$ is a proper subgraph of $G'$ by definition of $H$-critical. By Theorem~\ref{dPlanar}, we have that $d_{g,\varepsilon}(G'|H)\ge 0$. By Lemma~\ref{PurseDeletable}, it follows that $G'$ is connected. Note that $v(G)-v(G') +e(G)-e(G') < v(G)-v(H)+e(G)-e(H)$. Hence by induction, $d_{g,\varepsilon}(G|G')\ge 0$. By Proposition~\ref{dSum}, we have that

$$d_{g,\varepsilon}(G|H) = d_{g,\varepsilon}(G|G')+d_{g,\varepsilon}(G'|H)\ge 0 + 0 = 0,$$

\noindent as desired.
\end{proof}

\begin{theorem}\label{DeletableHyperbolic}
The family $\mathcal{F}_{345}$ is hyperbolic.
\end{theorem}
\begin{proof}
Let $c=\frac{6(5+\varepsilon)}{\varepsilon}$ where $\varepsilon$ is as in Theorem~\ref{dPlanar}. We prove that $c$ is a Cheeger constant for $\mathcal{F}_{345}$ as follows.

Let $G$ be a graph with rings $\mathcal{R}$ embedded in a surface $\Sigma$ of Euler genus $g(\Sigma)$ such that $G\in \mathcal{F}_{345}$, let $R$ be the total
number of ring vertices, and let $\xi : S_1 \rightarrow \Sigma$ be a closed curve that bounds an open disk $\Delta$ and intersects $G$ only in vertices. To avoid notational complications we will
assume that $\xi$ is a simple curve; otherwise we split vertices that $\xi$ visits more than once to reduce to this case. We may assume that $\Delta$ includes at least one vertex
of $G$, for otherwise there is nothing to show. Let $X$ be the set of vertices of $G$ intersected by $\xi$.

Let $G_0$ be the subgraph of $G$ consisting of all vertices and edges drawn in the closure of $\Delta$. We now regard $G_0$ as a graph embedded in the closure of $\Delta$ inside the plane. We define a planar graph $G_1$ obtained from $G_0$ as follows: we add a vertex $v$ in the complement of the closure of $\Delta$; for each vertex $x\in X$, we add the edge $vx$; we then subdivide each edge incident with $v$ exactly once. 

Note that $G_1$ is planar and has girth at least $g$. Let $H=G_1[\{x\}\cup N_{G_1}(x)\cup X]$. Note that $G_0-X$ is not $(8-g)$-deletable and hence $G_1-V(H)$ is not $(8-g)$-deletable. It follows from Lemma~\ref{PurseDeletable} that $G_1-V(H)$ is not a purse. Since $G_1$ is planar, it thus follows that $v(H)\ge 2$.

Since $v(H)\ge 2$, we have $v(H)\le 3|X| \le 6(|X|-1)$. Note that $H$ is connected. Hence by Lemma~\ref{dDeletable}, $d_{g,\varepsilon}(G_1|H) \ge 0$. Thus by Lemma~\ref{dBound}, $v(G_1) \le \frac{g+\varepsilon}{\varepsilon} v(H)$. But then

$$v(G_0) \le v(G_1) \le \frac{g+\varepsilon}{\varepsilon} v(H) \le \frac{6(5+\varepsilon)}{\varepsilon} (v(H)-1) = c(v(H)-1),$$

\noindent as desired.
\end{proof}

\subsection{Strong Hyperbolicity for No Deletable Subgraphs}\label{DeepStronglyHyperbolic}

We now prepare to prove the strong hyperbolicity of $\mathcal{F}_{345}$. First we prove the following lemma.

\begin{lemma}\label{dDeletableCylinder}
There exists $k>0$ such that the following holds: If $G$ is a connected planar graph having girth at least $g\in\{3,4,5\}$ and $H_1, H_2$ are non-empty vertex-disjoint connected subgraphs of $G$ such that there does not exist $X\subseteq V(G)\setminus (V(H_1)\cup V(H_2))$ such that $G[X]$ is $(8-g)$-deletable in $G$, then the distance between $H_1$ and $H_2$ in $G$ is at most $k(v(H_1)+v(H_2))$.
\end{lemma}
\begin{proof}
We prove that $k= k'\frac{5+\varepsilon}{\varepsilon}$ suffices where $k'$ is as in Corollary~\ref{OldLinear}. 

Suppose not. Let $H=H_1\cup H_2$. Let $G'$ be a subgraph of $G$ containing $H$ such that every component of $G'$ contains a vertex of $H$ and $d_{g,\varepsilon}(G'|H)\ge 0$ and subject to that, $e(G')+v(G')$ is maximum. 

Since $G$ is planar, we have that $G'$ is planar. Since $d_{g,\varepsilon}(G'|H)\ge 0$ and $G'$ is planar, we have by Lemma~\ref{dBound} that $v(G') \le \frac{g+\varepsilon}{\varepsilon} v(H)$. Note that it also follows from Lemma~\ref{dBound} that every component of $G'$ contains a vertex in $H$ and hence $G'$ has at most two components. 

First suppose that $V(G')=V(G)$. Then $v(G) = v(G') \le \frac{g+\varepsilon}{\varepsilon} v(H)$. Since $G$ is connected, it follows that the distance between $H_1$ and $H_2$ in $G$ is at most $v(G)\le \frac{5+\varepsilon}{\varepsilon} v(H) \le k'(v(H_1)+v(H_2))$, a contradiction.

So we may assume that $V(G')\ne V(G)$. Since $G-V(G')$ is not $(8-g)$-deletable in $G'$, we have by Proposition~\ref{DeletableToCritical} that there exists a subgraph $G''$ of $G$ containing $G'$ such that $G''$ is $G'$-critical for $(8-g)$-list-coloring.  

Now suppose that $G'$ is connected. By Lemma~\ref{dPlanar}, $d_{g,\varepsilon}(G''|G')\ge 0$. By Proposition~\ref{dSum}, 

$$d_{g,\varepsilon}(G''|H) = d_{g,\varepsilon}(G''|G') + d_{g,\varepsilon}(G'|H) \ge 0+0\ge 0,$$

\noindent contradicting the maximality of $G'$.

So we may assume that $G'$ is not connected. As noted above, we have that $G'$ has exactly two components $G_1$ and $G_2$. We may assume without loss of generality that $H_1\subseteq G_1$ and $H_2\subseteq G_2$. 

Next suppose that $G''$ is not connected. Since $G''$ is planar and $G''$ is $G'$-critical for $(8-g)$-list-coloring, it follows from Lemma~\ref{PurseDeletable} that every component of $G''$ contains a vertex of $G'$. Hence $G''$ has two components $G_1'$ and $G_2'$. We may assume without loss of generality that $G_1\subseteq G_1'$ and $G_2\subseteq G_2'$. Since $G' \subsetneq G''$, we may assume without loss of generality that $G_1\subsetneq G_1'$. But then $G_1'$ is $G_1$-critical for $(8-g)$-list-coloring. Hence by Theorem~\ref{dPlanar}, $d_{g,\varepsilon}(G_1'|G_1) \ge 0$. Let $G_3 = G_1'\cup G_2$. Now $d_{g,\varepsilon}(G_3|G') = d_{g,\varepsilon}(G_1'|G_1) \ge 0$. By Proposition~\ref{dSum}, we have that 

$$d_{g,\varepsilon}(G_3|H) = d_{g,\varepsilon}(G_3|G') + d_{g,\varepsilon}(G'|H) \ge 0+0\ge 0,$$

\noindent contradicting the maximality of $G'$.

So we may assume that $G''$ is connected. By Corollary~\ref{OldLinear}, we have that $v(G'')\le k'v(G') \le k'\frac{g+\varepsilon}{\varepsilon} v(H)$ where $k'$ is as in Corollary~\ref{OldLinear}. Since $H_1\cup H_2 = H\subseteq G'\subseteq G''$ and $G''$ is connected, it follows that the distance between $H_1$ and $H_2$ is at most $v(G'') \le  k'\frac{5+\varepsilon}{\varepsilon} (v(H_1)+v(H_2)) = k(v(G_1)+v(G_2))$, as desired.
\end{proof}

We are now ready to prove the strong hyperbolicity of $\mathcal{F}_{345}$ as follows.

\begin{proof}[Proof of Theorem~\ref{DeletableStronglyHyperbolic}]
Since $\mathcal{F}_{345}$ is hyperbolic by Theorem~\ref{DeletableHyperbolic}, it suffices to prove the existence of a strong hyperbolic constant for $\mathcal{F}_{345}$. In fact, we prove the following stronger statement:

\begin{claim} 
There exists $c'>0$ such that the following holds: Let $G\in \mathcal{F}_{345}$ be a graph embedded in a surface $\Sigma$ with rings $\mathcal{R}$. Let $D_1,D_2$ be two cycles of $G$ such that there exists a cylinder $\Lambda \subseteq \Sigma$ with boundary components $D_1$ and $D_2$. If $G'$ is the subgraph of $G$ consisting of all vertices and edges drawn in $\Lambda$, then 

$$v(G')\le c'(v(D_1)+v(D_2)).$$
\end{claim}
\begin{proof}
Specifically, we prove that $c' = \frac{5+\varepsilon}{\varepsilon}(k+1)$ suffices where $\varepsilon$ is as in Theorem~\ref{dPlanar} and $k$ is as in Lemma~\ref{dDeletableCylinder}.

Now we regard $G'$ as a planar graph. Let $H=D_1\cup D_2$. Note that $v(H)\le v(D_1)+v(D_2)$.

First suppose $H$ is connected. Then by Lemma~\ref{dDeletable}, $d_{g,\varepsilon}(G'|H) \ge 0$. Thus by Lemma~\ref{dBound}, $v(G') \le \frac{g+\varepsilon}{\varepsilon} v(H)$. But then

$$v(G') \le \frac{g+\varepsilon}{\varepsilon} v(H) \le \frac{(5+\varepsilon)}{\varepsilon} (v(D_1)+v(D_2)) \le c'(v(D_1)+v(D_2)),$$

\noindent as desired.

So we may assume that $H$ is not connected. First suppose that $G'$ is not connected. It follows from Theorem~\ref{PurseDeletable} that every component of $G'$ contains a vertex in $H'$. Hence $G'$ has exactly two components $G_1$ and $G_2$. We may assume without loss of generality that $G_1$ contains $D_1$ and $G_2$ contains $D_2$. By Lemma~\ref{dDeletable}, $d_{g,\varepsilon}(G_1|D_1) \ge 0$ and similarly $d_{g,\varepsilon}(G_2|D_2)\ge 0$. Thus by Theorem~\ref{dBound}, $v(G_1) \le \frac{g+\varepsilon}{\varepsilon} v(D_1)$ and $v(G_2) \le \frac{g+\varepsilon}{\varepsilon} v(D_2)$. But then

$$v(G') \le v(G_1)+v(G_2) \le \frac{5+\varepsilon}{\varepsilon} (v(D_1)+v(D_2)) \le c'(v(D_1)+v(D_2)),$$

\noindent as desired.

So we may assume that $G'$ is connected. Then by Lemma~\ref{dDeletableCylinder}, the distance between $D_1$ and $D_2$ is at most $k(v(D_1)+v(D_2)) \le k(v(D_1)+v(D_2))$. Let $P$ be a shortest path from $D_1$ to $D_2$. Let $H' = H\cup P$. Now $v(H')\le (k+1)(v(D_1)+v(D_2))$. Moreover, $H'$ is connected. Thus by Lemma~\ref{dDeletable}, $d_{g,\varepsilon}(G'|H') \ge 0$. Hence by Lemma~\ref{dBound}, $v(G') \le \frac{g+\varepsilon}{\varepsilon} v(H')$. But then

$$v(G') \le \frac{g+\varepsilon}{\varepsilon} v(H') \le \frac{5+\varepsilon}{\varepsilon} (k+1) (v(D_1)+v(D_2)) = c'(v(D_1)+v(D_2)),$$

\noindent as desired.
\end{proof}

\end{proof}

\section*{Acknowledgments}
I would like to thank Marthe Bonamy for helpful discussions about distributed coloring algorithms. I would like to thank Zden\v{e}k Dvo\v{r}\'ak for fruitful discussions about algorithms for coloring graphs on surfaces. I would like to thank Sergey Norin for insightful discussions about separators. Finally I would like to thank Robin Thomas for discussions about algorithms for coloring planar graphs and graphs on surfaces. I would also like to thank the anonymous referees for their helpful comments.

\end{document}